\documentclass[11pt]{amsart}

\usepackage{amssymb,upref}
\usepackage[mathcal]{euscript}

\newcommand{\bydef}{:=}

\newcommand{\cA}{\mathcal{A}}
\newcommand{\cE}{{\mathcal E}}

\newcommand{\bG}{\mathbf{G}}
\newcommand{\bmu}{\boldsymbol{\mu}}

\newcommand{\NN}{{\mathbb N}}

\newcommand{\ZZ}{{\mathbb Z}}
\newcommand{\FF}{\mathbb{F}}

\newcommand{\CC}{\mathbb{C}}

\newcommand{\bb}{\mathrm{b}}
\newcommand{\id}{\mathrm{id}}

\DeclareMathOperator{\Aut}{\mathrm{Aut}}
\DeclareMathOperator{\Der}{\mathrm{Der}}

\DeclareMathOperator{\Hom}{\mathrm{Hom}}
\DeclareMathOperator{\End}{\mathrm{End}}
\DeclareMathOperator{\Mat}{\mathrm{Mat}}
\DeclareMathOperator{\Lie}{\mathrm{Lie}}
\DeclareMathOperator{\charac}{\mathrm{char}}

\DeclareMathOperator{\AAut}{\mathbf{Aut}}
\DeclareMathOperator{\aAut}{\mathsf{Aut}}

\DeclareMathOperator{\DDiag}{\mathbf{Diag}}

\DeclareMathOperator{\Alg}{\mathrm{Alg}}

\newtheorem{theorem}{Theorem}[section]
\newtheorem{proposition}[theorem]{Proposition}
\newtheorem{lemma}[theorem]{Lemma}
\newtheorem{corollary}[theorem]{Corollary}

\theoremstyle{definition}
\newtheorem{df}[theorem]{Definition}
\newtheorem{example}[theorem]{Example}

\theoremstyle{remark}
\newtheorem{remark}[theorem]{Remark}

\usepackage{tikz}
\usetikzlibrary{arrows}

\begin{document}

\title{Evolution  algebras, automorphisms, and graphs }

\author[Alberto Elduque]{Alberto Elduque$^{\star}$}
\thanks{$^{\star}$ Supported by grants MTM2017-83506-C2-1-P (AEI/FEDER, UE) and E22\_17R (Diputaci\'on
General de Arag\'on).
Part of this research was done while this author was visiting the Departamento de Matem\'aticas, Facultad de Ciencias, Universidad de
Chile, supported by FONDECYT grant 1170547.}
\address{Departamento de Matem\'aticas e Instituto Universitario de Matem\'aticas y Aplicaciones,
Universidad de Zaragoza, 50009 Zaragoza, Spain
}
\email{elduque@unizar.es}

\author[Alicia Labra]{Alicia Labra$^{\star\star}$}
\thanks{$^{\star\star}$ Supported by FONDECYT 1170547.}
\address{Departamento de Matem\'aticas,
Facultad de Ciencias, Universidad de Chile.  Casilla 653, Santiago, Chile}
\email{alimat@uchile.cl}

\begin{abstract}
The affine group scheme of automorphisms of an evolution algebra $\cE$ with $\cE^2=\cE$ is shown to lie 
in an exact sequence $1\rightarrow \mathbf{D}\rightarrow \AAut(\cE)\rightarrow \mathsf{S}$, where $\mathbf{D}$, diagonalizable,
and $\mathsf{S}$, constant, depend solely on the directed graph associated to $\cE$.

As a consequence, the Lie algebra of derivations $\Der(\cE)$ (with $\cE^2=\cE$) is shown to be trivial if the characteristic of the ground field
is $0$ or $2$, and to be abelian, with a precise description, otherwise.
\end{abstract}

\maketitle


\section{Introduction}

Evolution algebras were introduced in 2006 by Tian and Vojtechovsky (see \cite{TV}) and present many connections with other fields like graph theory, group theory or Markov chains,
to mention a few (see Tian's monograph \cite{T}).  They have received  considerable attention in the  last years (see \cite{CSV} and the references therein).

In this paper, all the algebras considered will be defined over a ground field $\FF$, of arbitrary characteristic, and their dimension will be finite. An algebra is just a vector space $\cA$ endowed with a bilinear map (the multiplication) $\cA \times \cA \rightarrow \cA, (x,y) \mapsto xy$.

\begin{df}
An \emph{evolution algebra} is an algebra $\cE$ endowed with a basis $B =  \{v_1, v_2, \ldots, v_n\}, $ called a  \emph{natural basis}, such that $v_i v_j = 0$  for any  $1 \leq i \neq j \leq n$.
\end{df} 

Given any evolution algebra $\cE$ with natural basis  $B =  \{v_1, v_2, \ldots, v_n\}, $ and multiplication determined by
 $v_i^2 = \sum_{i = 1}^n \alpha_{ij} v_j $ ($\alpha_{ij} \in \FF$), an associated (directed) graph $ \Gamma= \Gamma (\cE,B)$ is defined in 
 \cite{EL}. The set of vertices $V$ of $\Gamma$ is just the natural basis, and the set of edges $E \subseteq V \times V$ consists of those pairs $(v_i,v_j) $ 
 with $\alpha_{ij} \neq 0$, that is, $(v_i, v_j) \in E$ if $v_j$ appears in $v_i^2$ with nonzero coefficient.
 
 The graph $ \Gamma= \Gamma (\cE,B)$ is used in \cite{EL,EL2} to get new results on these algebras and to provide new natural proofs of some known results. 
 
 In particular, it is shown in  \cite{EL} that the group of automorphisms $\Aut(\cE)$ is finite if $\cE^2 = \cE$ (or equivalently the matrix 
 $(\alpha_{ij})$  is regular). Over an infinite field $\FF$, the regular matrices form a Zariski open, and hence dense, set in $\Mat_n(\FF)$. So, in a way, we have that
 $\Aut(\cE) $ is finite for ``almost all'' evolution algebras.
 
 Over fields of positive characteristic, or over nonalgebraically closed fields of characteristic $0$, the affine group scheme of automorphisms $\AAut(\cE)$ contains much more information than $\Aut(\cE)$ including, in particular, the information on the Lie algebra of derivations $\Der(\cE)$.
 
 Here we follow the functorial approach to affine group schemes (see for instance \cite{W}). An affine group scheme  is a representable group-valued functor defined on the category $\Alg_{\FF}$ of unital commutative, associative algebras. Thus, given an evolution algebra $\cE$, $\AAut(\cE)$ is the functor $ \Alg_{\FF}\longrightarrow \mathrm{Grp}$ that takes any object $R$ in  $\Alg_{\FF}$ to the group $\Aut(\cE_R)$ of automorphisms, as an $R$-algebra, of $\cE_R \bydef
 \cE \otimes_{\FF} R$. The action on morphisms is the natural one.
 
 The Lie algebra $\Lie(\AAut(\cE))$ is canonically isomorphic to the Lie  algebra of derivations $\Der(\cE) = \{ \delta  \in \End_{\FF}(\cE) \mid \delta(xy) = \delta(x)y + x\delta(y)$ for any $x,y \in  \cE\}$ (see \cite[Example A.43]{EK}).
 
 Now, the fact that $\Aut(\cE)$ is finite if
 $\cE^2 = \cE $ \cite[Theorem 4.3]{EL} shows, in particular, that $\Aut(\cE_{\FF_{\text{alg}}})$ is finite, where $\FF_{\text{alg}}$ is an algebraic closure of $\FF$, and hence the affine group scheme $\AAut(\cE)$ is finite, that is, the Hopf algebra that represents it is finite dimensional.

If the characteristic of the ground field $\FF$ is $0$, then any finite affine group scheme is \'etale, and hence the Lie algebra is trivial. Therefore
\cite[Theorem 4.8]{EL} implies $\Der(\cE) = 0$ if $\cE^2 = \cE$ and $\charac(\FF) = 0$. 
(This result over $\CC$ has been proven in \cite[Theorem 2.6]{CGOT} in a different way).

However, as some examples in \cite{CMMS} show, this is no longer true if $\charac(\FF) > 0$. 

The goal of the present paper is to show that given any evolution algebra $\cE$ with  $\cE^2 = \cE$, there is an exact sequence \eqref{eq:exact}
\[
1  \longrightarrow  \DDiag(\Gamma) \longrightarrow   \AAut(\cE)  \longrightarrow  \aAut(\Gamma)  
\]
where $ \aAut(\Gamma)$ is the constant group scheme attached to the group of automorphism of the graph associated to $\cE$ in \cite{EL}, while 
$\DDiag(\Gamma)$ is a finite diagonalizable group scheme defined in terms solely of $\Gamma$. That is the elements in the exact sequence, except
$ \AAut(\cE)$ itself, depend only on $\Gamma$ (!!). 

An affine group scheme is diagonalizable if it is a subscheme of a torus \cite[\S 2.2]{W} or, equivalently, if the representing Hopf algebra is the gruop 
algebra of a finitely generated abelian group. In our situation, $\DDiag(\Gamma)$ turns out to be a product of schemes of roots of unity $\bmu_N$ ($N \in \NN$),
where $\bmu_N(R) = \{r \in R \mid r^N = 1\}$ for any $R$ in $\Alg_{\FF}$, which  is represented by the quotient $\FF[x] / (x^N-1)$, that is, the group algebra of the cyclic group of order $N$.

On the other hand, given  a finite group $G$, the associated constant group scheme $\mathsf{G}$ is the group scheme represented by $\FF^{G} \bydef \mathrm{Maps}(G,\FF)=\bigoplus_{g\in G}\FF \epsilon_g$, where 
\[ 
\epsilon_{g}(h) = \begin{cases}

 1  & \text{if $h = g$,}   \\

 0 & \text{otherwise,}  \\
 \end{cases}
\] 
(see  \cite[\S 2.4]{W}). For any $R$ in $\Alg_\FF$ without proper idempotents, $\mathsf{G}(R)$ is (isomorphic to) the group $G$.

Note that $\FF^G \simeq \FF \times \FF \times \cdots \times \FF$ is the cartesian product of $|G|$ copies of $\FF$.
In particular,  $\FF^G $ is a separable algebra and hence $\mathsf{G}$ is \'etale.

The paper is structured as follows. Section \ref{se:diagonal} will be devoted to define and study the diagonalizable affine group scheme $\DDiag(\Gamma)$ associated to any graph. For connected $\Gamma$, $\DDiag(\Gamma)$ is either trivial or isomorphic  to $\bmu_N$ for some natural number $N$, given by the so called \emph{balance} of $\Gamma$. Section \ref{se:graph_auto} will deal with the group of automorphisms of a graph. Its main result: Theorem \ref{th:exact}, gives the exact sequence
 \eqref{eq:exact} mentioned above. This exact sequence induces a short exact sequence \eqref{eq:short-exact} which does not split in general. 
 Finally Section 
 \ref{se:derivations} is devoted to describe the Lie algebra of derivations of any  evolution algebra $\cE$ with  $\cE^2 = \cE$. The description is a direct consequence of our results on the affine group scheme $ \AAut(\cE)$. It turns out that $\Der(\cE)$ depends only on the graph.

\smallskip
 

\section{The diagonal group of a graph}\label{se:diagonal}

All the graphs considered in this paper are directed graphs. These are pairs  $\Gamma =(V,E)$, consisting of a finite set of vertices $V$ and a set of edges (or arrows) 
$ E \subseteq V \times V$.

Given such a graph, we need some definitions
\begin{itemize}
\item A \emph{path} is a sequence $\gamma = (v_0, e_1, v_1, \ldots, v_{n-1}, e_n,v_n)$ where $n \geq 0$, $v_0, \ldots, v_n \in V$, $e_1, \ldots,e_n \in E$, and for each $i = 1,\ldots,n$, either $e_i =(v_{i-1}, v_i)$ or $e_i = (v_i, v_{i-1})$.

We define the   \emph{balance} of the path $\gamma $ as the integer 
\begin{multline*}
\bb(\gamma) = \Bigl\lvert\left\{i \mid 1 \leq i \leq n \ \mbox{and} \ e_i =(v_{i-1}, v_i )\right\}\Bigr\rvert  \\
- \Bigl\lvert\left\{i \mid 1 \leq i \leq n \ \mbox{and} \  e_i =(v_{i}, v_{i-1} )\right\}\Bigr\rvert \,.
\end{multline*}
that is, $\bb(\gamma)$ is obtained by adding $1$ if the edge $e_i$ goes in the ``right" direction (from $v_0$ to $v_n$) and $-1$ if the edge $e_i$ goes in the ``wrong" direction, and summing over $i$. 

The \emph{balance} of $\Gamma  $ is defined as the greatest common divisor of the absolute values of the balances of the cycles in $\Gamma$:
\[
\bb(\Gamma) = \gcd\left\{\lvert \bb(\gamma)\rvert : \gamma \ \mbox{cycle in} \ \Gamma \right\}.
\]

\item A \emph{cycle} is a path $\gamma = (v_0, e_1, v_1, \ldots, v_{n-1}, e_n,v_n)$ with $v_0 = v_n$.

\item The \emph{indegree} of a vertex $v$ is the natural number (or $0$)
\[
\deg^{-}(v) = \Bigl\lvert\left\{w \in V \mid (w,v) \in E\right\}\Bigr\rvert,
\]
while the \emph{outdegree} is 
\[ 
\deg^{+}(v) = \Bigl\lvert\left\{w \in V \mid (v,w) \in E\right\}\Bigr\rvert.
\]
The vertex $v$ is said to be a \emph{source} if $\deg^{-}(v) = 0$, and a \emph{sink} if $\deg^{+}(v) = 0$.

\item  $\Gamma $ is said to be \emph{connected} if the underlying undirected graph is connected, that is, if for every $v,w \in V$ there exists a path 
\[
\gamma = (v_0, e_1, v_1 ,\ldots, v_{n-1}, e_n,v_n)
\] 
with $v_0 = v $ and $ v_n = w$. Any graph $\Gamma $ is the ``disjoint union" of its 
\emph{connected components}
\end{itemize}

\begin{df}\label{def:diagonal}
The \emph{diagonal group} of a graph  $\Gamma =(V,E)$ is the (diagonalizable) affine group scheme $\DDiag(\Gamma)$ given by
\[
\DDiag (\Gamma)(R) = 
\{\varphi: V \longrightarrow R^{\times} \mid \forall  (v,w) \in E, \  \varphi(w) = \varphi(v)^2\},
\]
with the natural morphisms.
\end{df}

Note that  $\DDiag(\Gamma)$ is a subgroup scheme of the  torus $(\bG_m)^{|V|}$

\medskip

Let us see a few examples.

\begin{example}
\[
\Gamma:\qquad\raisebox{-2pt}{\begin{mbox}%
{\begin{tikzpicture}%
[->,>=stealth',shorten >=1pt,auto,thick,every node/.style={circle,fill=blue!20,draw}]

  \node (n1) at (4,4)  {};
  \node (n2) at (6,4)  {};
  \node (n3) at (8,4)  {};

   \path[every node/.style={font=\sffamily\small}]

           (n1) edge node [above] {$ $} (n2) ;
           
   \path[every node/.style={font=\sffamily\small}]
           
 (n3) edge node [above] {$ $} (n2) ;
\end{tikzpicture}}
\end{mbox}}
\]
 then 
\[
\DDiag(\Gamma)\simeq \bigl(\bG_m \times {\bG}_m\bigr) /  \{(\mu_1,\mu_2)  \mid \mu_1^2 = \mu_2^2\} \simeq \bG_m  \times \bmu_2.
\]
\end{example}

\medskip

\begin{example}
\[ 
\Gamma:\qquad\raisebox{-8pt}{\begin{mbox}%
{\begin{tikzpicture}%
[->,>=stealth',shorten >=1pt,auto,thick,every node/.style={circle,fill=blue!20,draw}]

  \node (n1) at (4,4)  {$a$};
  \node (n2) at (6,4)  {$b$};
  \node (n3) at (8,4)  {$c$};

   \path[every node/.style={font=\sffamily\small}]

           (n1) edge node [above] {$ $} (n2) ;
           
   \path[every node/.style={font=\sffamily\small}]

         (n3)  edge  [bend left] node[above] {$ $} (n2) ;

          \path[every node/.style={font=\sffamily\small}]
            (n2)  edge [bend left] node[below] {$ $} (n3) ;

\end{tikzpicture}}
\end{mbox}}
\] 
$\Gamma$ has no sinks.

 If $\varphi \in \DDiag(\Gamma)(R) $ and $\varphi(a) = \mu$ ($\in R^{\times}$),  then $\varphi(b) = \mu^2$, $\varphi(c) = \mu^4$, and $\varphi(b) = \varphi(c)^2$, that is $\mu^2 = \mu^8$, so $\mu^6 = 1$. Hence $\DDiag(\Gamma)\simeq \bmu_6$.
 
 \end{example}

\medskip

\begin{example}
\[
\Gamma:\qquad\raisebox{-8pt}{\begin{mbox}{%
\begin{tikzpicture}%
[->,>=stealth',shorten >=1pt,auto,thick,every node/.style={circle,fill=blue!20,draw}]

  \node (n1) at (4,4)  {$a$};
  \node (n2) at (6,4)  {$b$};
  \node (n3) at (8,4)  {$c$};

   \path[every node/.style={font=\sffamily\small}]

           (n2) edge node [above] {$ $} (n1) ;
           
   \path[every node/.style={font=\sffamily\small}]
   
         (n3)  edge  [bend left] node[above] {$ $} (n2);

       \path[every node/.style={font=\sffamily\small}]
            (n2)  edge [bend left] node[below] {$ $} (n3) ;

\end{tikzpicture}}
\end{mbox}}
\] 
$\Gamma$ has no sources.

Again, if $\varphi \in \DDiag(\Gamma)(R) $ and $\varphi(c) = \mu$, then $\varphi(b) = \mu^2$, $\varphi(a) = \mu^4$, and $\varphi(c) = \varphi(b)^2$, that is, $\mu = \mu^4$, so $\mu^3 = 1$. Hence $\DDiag (\Gamma)\simeq \bmu_3$.

\end{example}

\medskip

From the definitions, we get at once the next result:

\begin{proposition}\label{prop:connected}
Let $\Gamma =(V,E)$ be a graph with connected components $\Gamma_i =(V_i,E_i)$, $i = 1 , \ldots ,n$ (so that $ V = V_1 \mathop{\dot\cup} \cdots \mathop{\dot\cup} V_n$). Then 
\[
\DDiag(\Gamma) \simeq \DDiag(\Gamma_1)  \times \cdots \times \DDiag(\Gamma_n).
\]
 \end{proposition}

If $ m = 2s+1$ is an odd natural number the square map
\[ 
\begin{split}
\bmu_m(R) &\longrightarrow \bmu_m(R)\\
  r\ &\mapsto\  r^2,
\end{split}
\]
is a group automorphism for any $R$ in $\Alg_{\FF}$, with inverse $r\longrightarrow r^{\frac{1}{2}} := r^{s+1}$.
Therefore, expressions like $r^{2^{-3}}$ make sense: $r^{2^{-3}} = ((r^{\frac{1}{2}} )^{\frac{1}{2}})^{\frac{1}{2}}$.

\begin{lemma}\label{le:diagonal}
Let $\Gamma =(V,E)$ be a graph,  $\gamma = (v_0, e_1, v_1 ,\ldots, v_{n-1}, e_n,v_n)$  be a path in $\Gamma. $ Let $\varphi \in \DDiag(\Gamma)(R)$ for $R$ in $\Alg_{\FF}$,
such that $\varphi(v_i) \in \bmu_{m_i} (R)$ with $m_i$ odd for any $ i = 0, \ldots,n$. Then $\varphi(v_n) = \varphi(v_0)^{2^{\bb(\gamma)}}$.
\end{lemma}

\begin{proof}
Imagine that $\gamma = (v_0, e_1, v_1,e_2, v_2, e_3,v_3)$ with $e_1 = (v_1,v_0)$, $e_2 = (v_1,v_2)$, and $e_3 =(v_3,v_2)$, so $\bb(\gamma) = -1$.

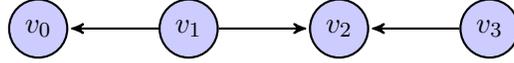
\[ 
\begin{tikzpicture}%
[->,>=stealth',shorten >=1pt,auto,thick,every node/.style={circle,fill=blue!20,draw}]

  \node (n1) at (4,4)  {$v_0$};
  \node (n2) at (6,4)  {$v_1$};
  \node (n3) at (8,4)  {$v_2$};
  \node (n4) at (10,4)  {$v_3$};
  
   \path[every node/.style={font=\sffamily\small}]

           (n2) edge node [above] {$ $} (n1) ;
           
   \path[every node/.style={font=\sffamily\small}]

         (n2)  edge  node[above] {$ $} (n3) ;

          \path[every node/.style={font=\sffamily\small}]
            (n4)  edge  node[below] {$ $} (n3) ;

\end{tikzpicture}
\] 
Then 
\begin{itemize}
\item As $e_1 = (v_1,v_0) \in E$, $\varphi(v_0) = \varphi(v_1)^2$, so $\varphi(v_1) = \varphi(v_0)^{\frac{1}{2}} = \varphi(v_0)^{2^{-1}}$.
\item As $ e_2 = (v_1,v_2) \in E$, $\varphi(v_2) = \varphi(v_1)^2$, so $\varphi(v_2) = (\varphi(v_0)^{2^{-1}})^2 = \varphi(v_0)$.
\item  As $ e_3 = (v_3,v_2) \in E$, $\varphi(v_2) = \varphi(v_3)^2$, so $\varphi(v_3) = \varphi(v_2)^{-1} = \varphi(v_1)^{2^{-1}} = \varphi(v_0)^{2^{\bb(\gamma)}}$.
\end{itemize}
 The general argument follows the same lines.
\end{proof}

Our next result determines the diagonal group of connected graphs without sources. Note that the graphs attached to evolution algebras $\cE$
with $\cE^2 = \cE$ have no sources.

\begin{theorem}\label{th:connected}
Let $\Gamma =(V,E)$ be a connected graph with no sources. Then $\DDiag(\Gamma) \simeq \bmu_N$ where $ N = 2^{\bb(\Gamma)} -1$.
\end{theorem}

\begin{proof}
First, the arguments  in the proof of \cite[Theorem 4.8]{EL} show that for any $R$ in $\Alg_{\FF}$,  any $\varphi \in \DDiag(\Gamma)(R)$, and any vector $v \in V$, $\varphi(v) \in  \bmu_{2^s -1}(R)$ for some natural number $s$.

Fix a vertex $a \in V$,  and consider the restriction homomorphism 
\[
\begin{split}
\Phi_a: \DDiag(\Gamma) &\longrightarrow \bG_m\\ 
\varphi\ &\mapsto\ \varphi(a).
\end{split}
\]
We will follow several steps:
\begin{itemize}
\item $\Phi_a$ is one-to-one.

Actually, for $R$ in $\Alg_{\FF}$,  and $\varphi \in \DDiag(\Gamma)(R)$,  with $\varphi(a) = 1$, by connectedness for any vertex $v \in V$ there is a path
$\gamma = (v_0, e_1, v_1, \ldots, v_{n-1}, e_n,v_n)$ with $v_0 = a$ and $ v_n = v$. By Lemma \ref{le:diagonal}, $\varphi(v) = \varphi(a)^{2^{\bb(\gamma)}} = 1^{2^{\bb(\gamma)}} 
=1$.

\smallskip

\item For any $R $ in $\Alg_{\FF}$, and $\varphi \in \DDiag(\Gamma)(R)$, $\varphi(a) \in\bmu_N(R)$. 

Indeed, by the previous argument, for any $v \in V$, $\varphi(v) = \varphi(a)^{2^{\bb(\gamma)}} $ for any path $\gamma$ connecting $a$ and $v$. As the order of $\varphi(a)$ is odd, $\varphi(a)$ and $\varphi(v)$  generate the same subgroup of $R^{\times}$.
In particular $\varphi(a)$ and $\varphi(v)$ have the same order.

Given any cycle $\gamma = (v_0, e_1, v_1, \ldots, v_{n-1}, e_n,v_n)$   in $\Gamma$ ($v_n = v_0$), we get $\varphi(v_0) = \varphi(v_0)^{2^{\bb(\gamma)}} $,  or 
$\varphi(v_0)^{2^{\bb(\gamma)}-1} =1$. Thus the order of $\varphi(a) $ divides $2^{|\bb(\gamma)|}-1$  for any cycle $\gamma$. Using that $2^{\gcd(m_1,m_2)} -1 =
\gcd(2^{m_1}-1, 2^{m_2}-1)$,  our result follows.

\smallskip

\item The image of $\Phi_a$ is exactly $\bmu_N$.

For any $R$  in $\Alg_{\FF}$,  and any $\mu \in \bmu_N(R)$, define $\varphi: V  \longrightarrow  R^{\times}$ as follows: For any  $v \in V$, select a path connecting $a$ and $v$: 
$\gamma = (v_0, e_1, v_1, \ldots, v_{n-1}, e_n,v_n)$ with $v_0 = a$ and $v_n = v$, and define $\varphi(v) = \mu^{2^{\bb(\gamma)}} $. This is well defined, because for any other path 
 $\hat{\gamma} = (\hat{v}_0, \hat{e}_1, \hat{v}_1,\ldots, \hat{v}_{n-1}, \hat{e}_n,\hat{v}_n)$ 
 connecting $a =\hat{v}_0$ and $v = \hat{v}_n$, then
\[ 
\gamma\hat{\gamma}^{-1} := (v_0, e_1, v_1 ,\ldots, v_{n-1}, e_n,v_n=\hat{v}_n, \hat{e}_n, \hat{v}_{n-1}, \ldots, \hat{e}_1, \hat{v}_0)
\]
is a cycle with balance $\bb(\gamma \hat{\gamma}^{-1} )= \bb(\gamma) -\bb(\hat{\gamma})$ and, therefore,
$\mu = \mu^{2^{ \bb(\gamma) -\bb(\hat{\gamma})}}$. Hence
\[  
\mu^{2^{\bb(\hat{\gamma})}}= (\mu^{2^{ \bb(\gamma) -\bb(\hat{\gamma})}})^{2^{\bb(\hat{\gamma})}} = \mu^{2^{\bb(\gamma)}}.
\]
Finally, $\varphi \in \DDiag(\Gamma)(R)$, because for any $e = (v,w) \in E, $ if $\gamma = (v_0, e_1, v_1 ,\ldots, v_{n-1}, e_n,v_n)$ is a path connecting $a = v_0$ and $v = v_n$, then
$\gamma' = (v_0, e_1, v_1 ,\ldots, v_{n-1}, e_n,v_n,e,w)$ is a path connecting $a$ and $w$ with $\bb(\gamma') = \bb(\gamma) +1$. Hence, 
$\varphi(w) = \mu^{2^{\bb(\gamma')}}= (\mu^{2^{\bb(\gamma)}})^2 = \varphi(v)^2$.
\end{itemize}
\end{proof}

\begin{corollary}
Let $\Gamma =(V,E)$ be a connected graph with no sources and with a loop $e=(v,v)$. Then $\DDiag(\Gamma)  =1. $
\end{corollary}

Let  $\cE$ be an evolution algebra with  natural basis $B=\{v_1, \ldots, v_n\}$  and let   $\Gamma = \Gamma(\cE,B) = (V,E)$ be the attached  graph $(V = B)$.  For any 
$R$ in $\Alg_{\FF}$,  and any $\varphi \in \DDiag(\Gamma)(R)$, $\varphi$ induces the linear (diagonal) isomorphism
\begin{equation}\label{eq:phitilde}
\begin{split}
\hat{\varphi}: \cE_R & \longrightarrow \cE_R\\  
v_i  &\mapsto \varphi(v_i)v_i. 
\end{split}
\end{equation}
Let $v_i^2 = \sum_{j=1}^{n} \alpha_{ij} v_j$ for $i =1, \ldots,n$, with $\alpha_{ij}\in\FF$, then 
\[
\hat{\varphi}(v_i^2) =  \sum_{j=1}^{n} \alpha_{ij} \hat{\varphi}(v_j) =  \sum_{j=1}^{n} \alpha_{ij} \varphi(v_j) v_j, 
\] 
and $\hat{\varphi}(v_i)^2 = \varphi(v_i)^2 \sum_{j=1}^{n} \alpha_{ij} v_j$.  

But if $\alpha_{ij} \neq 0$, then $(v_i,v_j) \in E$, so $\varphi(v_j) = \varphi(v_i)^2$. Hence $\hat{\varphi} \in \Aut(\cE_R)$ and we
obtain the following result:

\begin{theorem}\label{th:iota}
Let $\cE$ be an evolution algebra with natural basis $B$ and let $\Gamma = \Gamma(\cE,B)$ be the attached  graph. Then there is an injective homomorphism 
$\iota: \DDiag(\Gamma) \longrightarrow  \AAut(\cE)$ such that for any $R$ in $\Alg_{\FF}$,  and any $R$-point $\varphi \in \DDiag(\Gamma)(R)$, $ \iota (\varphi) = \hat{\varphi}$
(as in \eqref{eq:phitilde}).
\end{theorem}


\section{Graph Automorphisms}\label{se:graph_auto}


The goal of this section is, given an evolution algebra $\cE$ with $\cE^2 = \cE$ with attached graph $\Gamma(\cE,B)$ (which is independent, up to isomorphism, of the natural
 basis $B$ chosen \cite[Corollary 4.7]{EL}), to show the existence of a natural homomorphism
 \begin{equation}\label{eq:rho}
\rho: \AAut(\cE)  \longrightarrow \aAut(\Gamma)
\end{equation}
where $\aAut(\Gamma)$ is the constant group scheme attached to the group of automorphisms of $\Gamma$, denoted by $\Aut(\Gamma)$. If $B = \{v_1, \ldots, v_n\}$ is a natural basis we may identify $\Aut(\Gamma)$ with a subgroup of the symmetric group $S_n$ of degree $n$:
\[ 
\Aut(\Gamma) \simeq \left\{ \sigma \in S_n \mid \forall \; 1 \leq i,j \leq n,\ (v_i,v_j) \in E 
\Rightarrow (v_{\sigma(i)}, v_{\sigma(j)}) \in E \right\}.
\]
If we just look at the rational points in $\Aut(\cE) = \AAut(\cE)(\FF)$, any $\varphi \in \Aut(\cE)$ has an attached permutation $\sigma \in \Aut (\Gamma) $ such that 
$\varphi(v_i) \in \FF^{\times}v_{\sigma(i)}$ for any $i =1, \ldots,n$ (\cite[Theorem 4.4]{EL}). Thus the coordinate matrix of $\varphi$ relative to $B$ is a monomial matrix (i.e., it
has exactly one nonzero entry in each row and column). 
 In order to deal with the group scheme $\AAut(\cE) $, some extra care must be taken. Let $R$ be in  $\Alg_{\FF}$, and let $\varphi \in \AAut(\cE)(R) = \Aut(\cE_R)$, 
 with $\varphi(v_i) = \sum_{j = i}^{n} r_{ij} v_j$ for any $i =1, \ldots, n$.  Then $r = \det \bigl(r_{ij}\bigr) \in R^{\times}$:
 
\[
r = \sum_{\sigma \in S_n} (-1)^{\sigma } r_{\sigma(1) 1} \cdots r_{\sigma(n)n}\in R^\times .
\]
 
 For any $ i \neq j $ we have $ 0 = \varphi(v_i v_j) =  \varphi(v_i)\varphi(v_j) = \sum_{k=1}^n r_{ik}r_{jk} v_k^2$.
 
 Because  $\cE^2= \cE$, $\{v_1^2, \ldots, v_n^2\} $ form a basis of $\cE$ and hence 
\begin{equation}\label{eq:orthogonal}
r_{ik}r_{jk} = 0 \quad \text{for any} \  1 \leq i,j \leq n \ \mbox{with} \  i \neq j.
 \end{equation}
 
 Therefore, for any $\sigma \neq \tau $ in $S_n$,  
$ (r_{\sigma(1) 1} \cdots r_{\sigma(n)n}) ( r_{\tau(1) 1} \cdots r_{\tau(n)n}) = 0$.

For any $\sigma \in S_n$, consider the element 
\[
e_{\sigma}^{\varphi} = (-1)^\sigma r^{-1} r_{\sigma(1) 1} \cdots r_{\sigma(n)n}.
\] 
Then $ 1 = \sum_{\sigma\in S_n} e_{\sigma}^{\varphi}$, and $e_{\sigma}^{\varphi}e_{\tau}^{\varphi} =0$ for $\sigma \neq \tau$ in $S_n$. Therefore, the $e_{\sigma}^{\varphi}$'s are orthogonal idempotent elements, and 
$R = \bigoplus_{\sigma \in S_n} Re_{\sigma}^{\varphi}$. Moreover,   \eqref{eq:orthogonal} implies
\begin{equation}\label{eq:orthogonal2}
r_{ij}  e_{\sigma}^{\varphi}= 0 \quad \text{unless} \  i =  \sigma(j),
 \end{equation}
and the coordinate matrix $\bigl(r_{ij}\bigr) $ of $\varphi $ splits into a sum of monomial matrices over the orthogonal ideals $R e_{\sigma}^{\varphi}$. Thus, for instance, with $n = 3$ we have:
\begin{align*}
e_1^{\varphi } & = r_{11} r_{22}r_{33},& e_{(123)}^{\varphi }  & = r_{21} r_{32}r_{13}, & e_{(132)}^{\varphi }  & = r_{31} r_{12}r_{23}, \\
e_{(12)}^{\varphi } & =- r_{21} r_{12}r_{33},& e_{(23)}^{\varphi }  & = -r_{11} r_{32}r_{23}, & e_{(13)}^{\varphi }  & = -r_{31} r_{22}r_{13}.
\end{align*}
and $A = \bigl(r_{ij}\bigr)= \sum_{\sigma \in S_3} A_{\sigma}$,  with $A_{\sigma}= e_{\sigma}^{\varphi}A \in \Mat_3(Re_{\sigma}^{\varphi} )$ a monomial matrix thanks  to \eqref{eq:orthogonal2}:
\begin{align*}
A_1  &= e_1^\varphi 
 \begin{pmatrix}
r_{11}&0&0\\
0&r_{22}&0\\
0& 0 &r_{33}\\
\end{pmatrix},
& 
 A _{(123)}  &=e_{(123)}^\varphi 
\begin{pmatrix}
0&0&r_{13}\\
r_{21} &0&0\\
0& r_{32}&0\\
\end{pmatrix},
\\
 A _{(132)}  &=e_{(132)}^\varphi 
\begin{pmatrix}
0& r_{12}&0\\
0& 0&r_{23}\\
r_{13}& 0&0\\
\end{pmatrix}, 
&
 A _{(12)}  &=e_{(12)}^\varphi 
\begin{pmatrix}
0&r_{12}&0\\
r_{21}&0&0\\
0& 0 &r_{33}\\
\end{pmatrix},
\\
 A _{(23)}  &=e_{(23)}^\varphi 
\begin{pmatrix}
 r_{11}&0& 0\\
0& 0&r_{23}\\
0& r_{32}& 0\\
\end{pmatrix}, 
& 
A _{(13)}  &=e_{(13)}^\varphi 
\begin{pmatrix}
0& 0&r_{13}\\
0&r_{22}&0\\
r_{31}&0&0\\
\end{pmatrix}.
\end{align*}

Moreover, if $\sigma\in S_n$ and $ e_{\sigma}^{\varphi} \neq 0$, then the monomial matrix 
\[ 
A_{\sigma} = e_{\sigma}^{\varphi}\bigl(r_{ij}\bigr) = \sum_{i=1}^{n}  e_{\sigma}^{\varphi} r_{\sigma(i)i}E_{\sigma(i)i} ,
\] 
where $E_{ij} $ denotes the matrix with $1$ in the $(ij)$ slot and $0$'s elsewhere, correspond to an automorphism of $\cE_{R e_{\sigma}^{\varphi}}$. This forces
$ \sigma \in \Aut(\Gamma)$. Therefore, 
\begin{equation}\label{eq:orthogonal3}
e_{\sigma}^{\varphi} \neq 0 \ \text{only if} \  \sigma \in \Aut(\Gamma),\qquad 1 = \sum_{\sigma \in \Aut(\Gamma)} e_{\sigma}^{\varphi}\; .
\end{equation}

Recall that the coordinate Hopf algebra of the constant group scheme $\aAut(\Gamma) $ is $\FF^{\Aut(\Gamma)} = \mathrm{Maps} \bigl(\Aut(\Gamma), \FF\bigr)$, which has a natural basis $\{\epsilon_{\sigma} \mid \sigma \in \Aut(\Gamma)\}$, with 
\[ 
\epsilon_{\sigma}(\tau) = \begin{cases}

 1  & \; \mbox{if}\ \sigma =\tau,    \\

 0 & \; \mbox{otherwise.} \\
 \end{cases}
\] 
Then $\aAut(\Gamma)(R)$ is identified with $\Hom_{\Alg_{\FF}}(\FF^{\Aut(\Gamma)}, R)$.  

We are ready to define the homomorphism $\rho$ in \eqref{eq:rho}. For $R$ in $\Alg_{\FF}$ and $\varphi \in \AAut(\cE)(R) = \Aut (\cE_{R})$, the image of $\varphi$ under $\rho$ is defined as the element $\rho(\varphi) \in \Hom_{\Alg_{\FF}}(\FF^{\Aut(\Gamma)}, R)$ given by

\begin{equation}\label{eq:rho-phi}
\begin{split}
\rho(\varphi): \FF^{\Aut(\Gamma)} & \longrightarrow R\\  
\epsilon_{\sigma}\  &\mapsto\ e_{\sigma}^{\varphi}. 
\end{split}\end{equation}

It is trivially checked that this gives a homomorphism $\rho: \AAut(\cE) \rightarrow \aAut(\Gamma)$.

\begin{remark}
Exactly as over $\FF$,  if $R$ in $\Alg_{\FF} $ has no proper idempotents, then $ 1 = e_{\sigma}^{\varphi}$ for a unique $\sigma \in \Aut(\Gamma)$ and the matrix of $\varphi$ is a monomial matrix attached to $\sigma$. In this case $\aAut(\Gamma)(R) \simeq \Aut(\Gamma)$ and $\rho(\varphi) $ is just $\sigma$ under this identification.
\end{remark}

The main result of this section is the following:

\begin{theorem}\label{th:exact}
Let  $\cE$ be an evolution algebra with $\cE^2 =\cE$ and  natural basis $B=\{v_1, \ldots, v_n\}$.  Let   $\Gamma = \Gamma(\cE,B) $ be its associated  graph. Then the sequence
\begin{equation}\label{eq:exact}
1  \longrightarrow  \DDiag(\Gamma) \stackrel{\iota}{\longrightarrow }  \AAut(\cE)  \stackrel{\rho}{\longrightarrow } \aAut(\Gamma)
\end{equation}
is exact.
\end{theorem}

\begin{proof}
$\ker(\rho)(R)$ consists of the automorphisms $\varphi \in \AAut(\cE)(R) = \Aut(\cE_R)$ such that $ e_{\sigma}^{\varphi} = 0$ for any $1 \neq \sigma \in \Aut(\Gamma)$. Hence $1=e_1^\varphi$ and $\varphi$ is diagonal, that is, the elements of $B$ are eigenvectors for $\varphi$. These automorphisms are precisely the elements in the image of $\iota$.
\end{proof}

\begin{example}\label{ex:aut}
The homomorphism $\rho$ is not surjective in general. Take, for instance, the evolution algebra  $\cE= \FF v_1\oplus  \FF v_2  $, with natural basis $B = \{v_1,v_2\}$, and multiplication given by $v_1^2 = v_1+ \alpha v_2$, $v_2^2 = \beta v_1 +  v_2$, with $0 \neq \alpha, \beta \in \FF$, $\alpha\neq \beta$, $ \alpha \beta \neq 1 $. Then the associated graph $\Gamma(\cE, B)$ is the complete graph
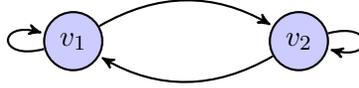
\[
\begin{tikzpicture}
[->,>=stealth',shorten >=1pt,auto,thick,every node/.style={circle,fill=blue!20,draw}]
  \node (n1) at (4,8)  {$v_1$};
  \node (n2) at (7,8) {$v_2$};

   \path[every node/.style={font=\sffamily\small}]
           (n2)  edge [bend left] node[below] {$ $} (n1)
       (n2) edge [loop right] node {$ $} (n2)
        (n1)  edge  [bend left] node[above] {$ $} (n2)
       (n1) edge [loop left] node {$ $} (n1);

\end{tikzpicture}
\]
While $\Aut(\Gamma) = C_2$, let us check that $\AAut(\cE) = 1$. To do that, it is enough to prove that $\Aut(\cE_R) = 1$ for $R$ in $\Alg_{\FF}$ without proper idempotents.

The arguments above show that the coordinate matrix relative to $B = \{v_1,v_2\}$  of any $\varphi \in \Aut(\cE_R)$ is either 
\[
\begin{pmatrix}
r_1& 0\\
0&r_2\\
\end{pmatrix} \qquad \mbox{or} \qquad 
\begin{pmatrix}
0&r_1\\
r_2 & 0\\
\end{pmatrix} \;,
\]
with $r_1, r_2 \in R^{\times}$.

In the first case $\varphi(v_1^2) = r_1 v_1 + \alpha r_2 v_2$, while $\varphi(v_1)^2 = r_1^2( v_1 + \alpha v_2)$, so $r_1^2 = r_1 = r_2$, and hence, due to the absence of 
proper idempotents, $\varphi = \id$.

In the second case $\varphi(v_1^2) = r_1 v_2 + \alpha r_2 v_1$, while $\varphi(v_1)^2 =r_1^2  v_2^2  =r_1^2( \beta v_1 +  v_2)$, so $r_1^2 = r_1$  and $\alpha r_2 = 
\beta r_1^2$.  Hence, $r_1 = 1$, $r_2 = \beta \alpha^{-1} \neq 1$. But $\varphi (v_2^2)= \varphi (v_2)^2$ forces $r_2 = 1$, a contradiction.
\end{example}

Any subgroup scheme of a constant group scheme is itself a constant group scheme. Hence we have the next consequence:

\begin{corollary}\label{co:exact}
Let  $\cE$ be an evolution algebra with $\cE^2 =\cE$ and  natural basis $B=\{v_1, \ldots, v_n\}$.  Let   $\Gamma = \Gamma(\cE,B) $ be its associated  graph. Then  there is a subgroup 
$H$ of $\Aut(\Gamma)$ and a short exact sequence
\begin{equation}\label{eq:short-exact}
1  \longrightarrow  \DDiag(\Gamma) \stackrel{\iota}{\longrightarrow }  \AAut(\cE)  \stackrel{\rho}{\longrightarrow }  \mathsf{H}  \longrightarrow  1,
\end{equation}
where $\mathsf{H}$ is the constant group scheme associated to $H$.
\end{corollary}

\begin{example}\label{ex:non-split}
The short exact sequence in Corollary \ref{co:exact} does not split in general.  Take, for instance the evolution algebra  $\cE= \FF v_1\oplus  \FF v_2  $  with  
$v_1^2 = v_2$, $v_2^2 = \alpha v_1 $, with $0\neq\alpha \in\FF$.
The associated graph is
\[ 
\qquad\raisebox{-8pt}{\begin{mbox}%
{\begin{tikzpicture}%
[->,>=stealth',shorten >=1pt,auto,thick,every node/.style={circle,fill=blue!20,draw}]

  \node (n2) at (4,8)  {$v_1$};
  \node (n3) at (7,8)  {$v_2$};


           
   \path[every node/.style={font=\sffamily\small}]

         (n3)  edge  [bend left] node[above] {$ $} (n2) ;

          \path[every node/.style={font=\sffamily\small}]
            (n2)  edge [bend left] node[below] {$ $} (n3) ;

\end{tikzpicture}}
\end{mbox}}
\] 
 Then $\DDiag(\Gamma) = \bmu_3$ (Theorem \ref{th:connected}) and $ \rho:  \AAut(\cE) \longrightarrow \AAut(\Gamma) \simeq \mathsf{C}_2$ is surjective, as it is so over an algebraic closure $\FF_{\text{alg}}$. Indeed,  over $\FF_{\text{alg}}$  the assignment 
\begin{equation}\label{eq:alpha13}
v_1 \mapsto  \alpha^{-1/3} v_2, \quad v_2\mapsto \alpha^{1/3} v_1,
\end{equation}
gives an automorphism $\varphi$ with $\rho(\varphi)$ being the generator of $\Aut(\Gamma)$. Moreover, $\varphi^2 = \id$ and this proves that \eqref{eq:short-exact} splits over $\FF_{\text{alg}}$. 

Let us check that the short exact sequence 
\begin{equation}\label{eq:short-exact2}
1  \longrightarrow  \DDiag(\Gamma) \longrightarrow   \AAut(\cE)  \longrightarrow  \aAut(\Gamma)\longrightarrow  1
\end{equation}
splits if and only if there is $\mu \in \FF$ such that $ \alpha = \mu^3$.

Actually, if  $ \alpha = \mu^3$ the assignment \eqref{eq:alpha13} makes sense over $\FF$, so the sequence splits. 
Conversely, if \eqref{eq:short-exact2} splits, there is an automorphism $\varphi \in \Aut(\cE)$ with $\varphi^2 = \id$, such that $\varphi(v_1) \in \FF^{\times}v_2$, $\varphi(v_2) \in \FF^{\times}v_1$.
With $\varphi(v_1) = \nu v_2$, $\varphi(v_2) = \mu v_1$, we get $\nu =\mu^{-1} $, as   $\varphi^2 = \id$, and
\[
\mu v_1 = \varphi(v_2) = \varphi(v_1^2)  = \varphi(v_1)^2 = \mu^{-2}v_2^2 = \mu^{-2}\alpha v_1,
\]
so that $\alpha = \mu^{3}$.  \qed
\end{example}


\section{Derivations}\label{se:derivations}


The results of the previous sections allow us to compute easily the Lie algebra of derivations of any evolution algebra 
$\cE$,  with $\cE^2 = \cE$. This Lie algebra depends only on the associated graph!

\begin{theorem}\label{th:Der}
Let $\cE$ be an evolution algebra  with $\cE^2 = \cE$. Let $B$ be a natural basis and let $\Gamma = \Gamma(\cE,B) $ be the attached  graph. Then:
\begin{enumerate}
\item If the characteristic of $\FF$ is $0$ or $2$, then $\Der(\cE) = 0$.
\item  If the characteristic of $\FF$ is $p \neq 0,2$, then $\Der(\cE) $ is an abelian Lie algebra whose dimension is the number of connected components
$\Gamma_i $  of $\Gamma$ such that the order of $2$ in $\ZZ/ p \ZZ$ divides the balance $\bb(\Gamma_i)$. 
\end{enumerate}
\end{theorem}

\begin{proof}
The exact sequence \eqref{eq:exact} induces an exact sequence (see eq. [Milne, 10d]):
\[
0  \longrightarrow \Lie(\DDiag(\Gamma)) \stackrel{\textup{d}\iota}{\longrightarrow } \Lie(\AAut(\cE))  \stackrel{\textup{d}\rho}{\longrightarrow } \Lie(\aAut(\Gamma))
\]
But $\Lie(\aAut(\Gamma)) = 0$, as $\aAut(\Gamma)$ is a constant group scheme, and hence \'etale. On the other hand, $\Lie(\AAut(\cE)) = \Der(\cE)$
(see \cite[Example A.43]{EK}), so that $\Der(\cE)$ is isomorphic to $ \Lie (\DDiag(\Gamma))$ through the differential of $\iota$.
 
However, $\Lie(\bmu_m)$ is either $0$ if $\charac(\FF) \nmid m$, or it has dimension $1$ if  $\charac(\FF) \mid m$ (see \cite[Example A42]{EK}). Hence
Theorem \ref{th:connected} gives the results.
\end{proof}

\begin{remark}
As mentioned in the Introduction, the fact that $\Der(\cE)$ is $0$ for any evolution algebra $\cE$ with $\cE^2 = \cE$ over $\CC$ has already been proved in 
\cite[Theorem 2.1]{CGOT}.
\end{remark}

Consider the algebra of dual numbers $\FF[\epsilon] = \FF 1 \oplus \FF \epsilon$,
with $\epsilon^2 = 0$, and the natural homomorphism $\pi: \FF[\epsilon]\longrightarrow \FF$ in $\Alg_\FF$ ($\pi(1)=1$, $\pi(\epsilon)=0$).
Given a graph $\Gamma = (V,E)$, $\Lie(\DDiag(\Gamma))$ is the kernel of the induced group homomorphism $\pi_*:\DDiag(\Gamma)(\FF[\epsilon])  \longrightarrow \DDiag(\Gamma)(\FF)$. The elements of $\ker\pi_*$ are the maps
\[
\begin{split}
\varphi: V & \longrightarrow \FF[\cE]\\  
v  &\mapsto 1 + \delta(v) \epsilon,
\end{split}
\]
for a linear map $\delta: V \longrightarrow \FF$, such that, for any $(v,w) \in E$, $\varphi(w) = \varphi(v)^2$, which is equivalent to $\delta(w) = 2\delta(v)$.

Therefore we obtain the following straightforward  consequence of Theorems \ref{th:Der}, \ref{th:connected} and \ref{th:iota}.

\begin{corollary}
Let $\cE$ be an evolution algebra  with $\cE^2 = \cE$ over a field $\FF$ of characteristic $p \neq 0,2$. Let $B$ be a natural basis and let 
$\Gamma = \Gamma(\cE,B) $ be the associated graph. Let $\Gamma_i = (V_i, E_i)$ ($V_i\subseteq B$), $i = 1, \cdots, r $, be the connected components
of $\Gamma $ such that $p \mid 2^{\bb(\Gamma_i) }-1$. For any $i = 1, \ldots, r $, fix an element $v_i \in V_i$. Then a basis of $ \Lie (\DDiag(\Gamma))$ is given by 
$\hat{\delta}_1, \cdots, \hat{\delta}_r$, where
\begin{itemize}
\item $\hat{\delta}_i(v) = 0 $ if  $v \notin V_i$,
\item $\hat{\delta}_i(v_i) = v_i$,
\item $\hat{\delta}_i(w) = 2^{\bb(\gamma)} w$ if $w \in V_i$ and $\gamma = (w_0,e_1,w_1, \ldots, e_n,w_n)$ is a path connecting $w_0 = v_i$ and $w_n = w$.
\end{itemize}
\end{corollary}

\begin{example}
The evolution algebra $\cE$ in Example \ref{ex:aut} has trivial group scheme of automorphisms, so $\Der(\cE) = 0$ for any ground field $\FF$.

However, for the  evolution algebra $\cE$ in Example \ref{ex:non-split}, we have the short exact sequence in \eqref{eq:short-exact2}, and $\DDiag(\Gamma) \simeq \
\bmu_3$. Hence $\Der(\cE) = 0$ unless $\charac(\FF) = 3$. In the later case, $\Der(\cE)$ is spanned by the map $d: v_1  \mapsto v_1$, $v_2 \mapsto 2v_2 = -v_2$. 
\end{example}

\begin{remark}
It must be remarked that for $\alpha = 1$,  the evolution algebra $\cE$ in Example \ref{ex:non-split} is the two-dimensional split para-Hurwitz algebra, and hence, for arbitrary $\alpha $ ($\neq 0$), $\cE $ is  a symmetric composition algebra (see \cite{E} and references therein).

As shown in Example \ref{ex:non-split}, the short exact sequence 
\[ 
1 \longrightarrow \bmu_3 \longrightarrow \aAut(\cE)\longrightarrow \mathsf{C}_2 \longrightarrow 1
\]
splits if and only if $\alpha  \in \FF^3$, that is, if and only if $\cE$ is, up to isomorphism, the split two-dimensional para-Hurwitz algebra.
\end{remark}

\end{document}